\documentclass[a4,12pt]{amsart}

\usepackage{setspace} 
\usepackage{amsmath,amssymb,amsthm}
\usepackage[dvipdfmx]{graphicx}
	\graphicspath{{./fig/}}

\allowdisplaybreaks
  
\makeatletter

\@addtoreset{equation}{section}
\makeatother

\newcommand{\C}{\mathbb{C}}

\newcommand{\Z}{\mathbb{Z}}
\newcommand{\I}{\sqrt{-1}}

\newcommand{\e}{\mathbf{e}}
\newcommand{\fourCtwo}{\frac{[4][3]}{[2]}}
\newcommand{\IfourCtwo}{\frac{[4][3]}{[2]} \I}

\newcommand{\mat}[4]{
\begin{pmatrix}
#1 & #2 \\
#3 & #4 
\end{pmatrix}
}

\theoremstyle{plain}
   \newtheorem{thm}{Theorem}[section]
   \newtheorem{prop}[thm]{Proposition}
   \newtheorem{lem}[thm]{Lemma}
   \newtheorem{cor}[thm]{Corollary}
\theoremstyle{definition}

\newcommand{\cmtout}[1]{#1}
\renewcommand{\cmtout}[1]{ }

\title[The $E_6$ state sum invariant of lens spaces]
{The $E_6$ state sum invariant of lens spaces}
\author{Kenta Okazaki}
\address{Research Institute for Mathematical Sciences,
Kyoto University,
Sakyo-ku, Kyoto, 606-8502,
Japan}
\email{junes@kurims.kyoto-u.ac.jp}
\date{}
\begin{document}
\maketitle

\begin{abstract}
In this paper, we calculate the values of the $E_6$ state sum invariants for the lens spaces $L(p,q)$.
In particular, we show that
the values of the invariants are determined by 
$p \mod 12$ and $q \mod (p,12)$.
As a corollary, we show that the $E_6$ state sum is a homotopy invariant for the oriented lens spaces.
\end{abstract}

\section{Introduction}
In \cite{TV}, Turaev and Viro constructed a state sum invariant of 3-manifolds
based on their triangulations,
by using the $6j$-symbols of representations of 
the quantum group $U_q(\mathfrak{sl}_2)$.
Further, Ocneanu \cite{Ocneanu2}
generalized the construction
to the case of other types of $6j$-symbols, say,
the $6j$-symbols of subfactors.
{\em The $E_6$ state sum invariant}
is the state sum invariant constructed from the $6j$-symbols of the $E_6$ subfactor,
which we denote by $Z$.
Suzuki and Wakui \cite{SW} calculated the $E_6$ state sum invariant for some of the lens spaces, where they used the representation of the mapping class group of a torus $SL(2,\Z)$.

\vskip 0.3pc

In this paper, we calculate the $E_6$ state sum invariants for all of the lens spaces, as follows.
For integers $m,n$, we denote by $(m,n)$ the great common devisor of $m$ and $n$.
We put $\zeta=\exp(\pi\I/12)$ and
$ [n] = (\zeta^n-\zeta^{-n})/(\zeta-\zeta^{-1})$ for an integer $n$,
noting that
\begin{align*}
	&[12-n] = [n], \quad [n+12] = -[n], \\
	&[2] =(1+\sqrt3)/\sqrt2, \quad
	[3] = 1+\sqrt{3}, \quad
	[4] =  (3 +\sqrt3)/\sqrt2.
\end{align*}

\begin{thm} \label{thm}
	For coprime integers $p$ and $q$, 
the $E_6$ state sum invariant of the lens space $L(p,q)$ is given as
\begin{align} \label{lensvalue}
	Z(L(p,q))=
\begin{cases}
 	|[p]| & \text{if \ } (p,12)=1, \\
 	[4][3]/[2] & \text{if \ } (p,12)=2,6, \\
 	\zeta^{\pm3}[4] & \text{if \ } (p,12)=3 \text{ and } q\equiv\pm1 \mod 3, \\
 	2\zeta^{\pm2} [3] & \text{if \ } (p,12)=4 \text{ and } q\equiv\pm1 \mod 4, \\
 	2[4][3]/[2] & \text{if \ } 12|p \text{ and } q\equiv \pm 1 \mod12, \\
 	0 & \text{if \ } 12|p\text{ and } q\equiv \pm 5 \mod12.
\end{cases}
\end{align}
In particular, the value of $Z(L(p,q))$ is determined by $p \mod 12$ and $q \mod (p,12)$.
\end{thm}
We note that we normalize the invariant so that $Z(S^3)=1$. 
Thus, our $Z$ is equal to $w Z$ in \cite{SW}, where we put  $w=2+[3]^2=6+2\sqrt3$.
\begin{cor} \label{cor}
	If there exists an orientation-preserving homotopy equivalence
	between the two lens spaces $L(p,q)$ and $L(p',q')$,
	then $Z(L(p,q))=Z(L(p',q'))$.
\end{cor}

\thanks{
The author would like to thank Tomotada Ohtsuki for his encouragement and comments.
 The author is also grateful to Michihisa Wakui for his helpful comments.
}

\section{The calculation of the $E_6$ state sum invariant}
In this section,
we briefly review the calculation of the $E_6$ state sum invariants for the lens spaces.
Suzuki and Wakui \cite{SW} defined the representation 
$\rho: SL(2,\Z) \longrightarrow GL_{10}(\C)$ by
{\tiny
\begin{align*}
&\rho(S)
\\
&=\frac1w
\begin{pmatrix}
	1 & [3] & 1 & [2]^2 & [3] & [3] & \fourCtwo & [3] & [3] & [2]^2 \\
	[3] & \IfourCtwo & -[3] & -[3] & 0 & -\IfourCtwo & 0 & [3] & -[3] & [3] \\
	1 & -[3] & 1 & [2]^2 & -[3] & -[3] & -\fourCtwo & [3] & [3] & [2]^2 \\
	[2]^2 & -[3] & [2]^2 & 1 & -[3] & -[3] & \fourCtwo & -[3] & -[3] & 1 \\
	[3] & 0 & -[3] & -[3] & 0 & 0 & 0 & -2 [3] & 2 [3] & [3] \\
	[3] & -\IfourCtwo & -[3] & -[3] & 0 & \IfourCtwo & 0 & [3] & -[3] & [3] \\
	\fourCtwo & 0 & -\fourCtwo & \fourCtwo & 0 & 0 & 0 & 0 & 0 & -\fourCtwo \\
	[3] & [3] & [3] & -[3] & -2 [3] & [3] & 0 & [3] & [3] & -[3] \\
	[3] & -[3] & [3] & -[3] & 2 [3] & -[3] & 0 & [3] & [3] & -[3] \\
	[2]^2 & [3] & [2]^2 & 1 & [3] & [3] & -\fourCtwo & -[3] & -[3] & 1
\end{pmatrix},
\\
&\rho(T)=
\begin{pmatrix}
	1 & 0 & 0 & 0 & 0 & 0 & 0 & 0 & 0 & 0  \\
	0 & -\zeta^2 & 0 & 0 & 0 & 0 & 0 & 0 & 0 & 0  \\
	0 & 0 & -1 & 0 & 0 & 0 & 0 & 0 & 0 & 0  \\
	0 & 0 & 0 & 1 & 0 & 0 & 0 & 0 & 0 & 0  \\
	0 & 0 & 0 & 0 & \I & 0 & 0 & 0 & 0 & 0  \\
	0 & 0 & 0 & 0 & 0 & -\zeta^2 & 0 & 0 & 0 & 0  \\
	0 & 0 & 0 & 0 & 0 & 0 & 1 & 0 & 0 & 0  \\
	0 & 0 & 0 & 0 & 0 & 0 & 0 & \zeta^8 & 0 & 0  \\
	0 & 0 & 0 & 0 & 0 & 0 & 0 & 0 &  \zeta^{-4} & 0  \\
	0 & 0 & 0 & 0 & 0 & 0 & 0 & 0 & 0 & -1  \\
\end{pmatrix},
\end{align*}
}
\noindent
where $S=\mat0{-1}10$ and $T=\mat1101$ are generators of $SL(2,\Z)$.

Let $p,q$ be coprime integers.
It is known, see \cite[Lemma 4.2]{SW},
that the $E_6$ state sum invariants of lens spaces are given as
\begin{align} \label{eq.SW}
	Z(L(p,q))
	&=
	w {}^t \e \rho(\mat {-q}{b}p{-a} ) \e,
\end{align}
where $a$, $b$ are integers satisfying $aq-bp=1$, 
and we put $\e ={}^t (1,0,0,0,0,0,0,0,0,0).$
We note that the right-hand side of the above formula does not depend on the choice of $a$ and $b$.

\section{Proof of the theorem}
In this section, we prove Theorem \ref{thm} and Corollary \ref{cor}.
In order to show Theorem \ref{thm}, 
we show Proposition \ref{12periodic}, which says that the values of $Z(L(p,q))$ have period $12$ for $p$ and $q$ . 
In order to show Proposition \ref{12periodic}, we show Lemmas
\ref{lem.abab} and \ref{trivact}, as follows.

\begin{lem} \label{lem.abab}
Let $p,q,p',q'$ be integers satisfying
$(p,q)=1$, $(p',q')=1$ and $p \equiv p' ,q \equiv q' \mod12$. 
Then, there exist integers $a,b,a',b'$ such that 
\[
\text{$aq-bp=1$, $a'q'-b'p'=1$\quad and\quad $a \equiv a'$, 
$b \equiv b' \mod 12$.}
\]
\end{lem}

\begin{proof}
We put integers $a$ and $b$ satisfying $aq-bp=1$.
Further, we put 
\[
	\mat{a'}{p'}{b'}{q'} = \mat{a}{p}{b}{q}  + 12 \mat xzyw,
\]
	noting that, by assumption, $z$ and $w$ are determined uniquely.
	It is sufficient to show that there exist integers $x$ and $y$ satisfying $a'q'-b'p'=1$.
	The determinant of the right-hand side of the above formula is equal to
\begin{align} \label{deter}
	&(a+12x)(q+12w)-(p+12z)(b+12y)
	\\&= \notag
	a(q+12w)+12xq'-(p+12z)b-12p'y
	\\&= \notag
	1+12(aw+xq'-zb-p'y).
\end{align}
	Since $(p',q')=1$, there exists integers $x$ and $y$ satisfying
\[p' y-q'x=aw-zb.\]
Then, the last term of (\ref{deter}) is equal to $1$.
Thus, we have $a'q'-b'p'=1$, as required.
\end{proof}

\noindent
We denote by $I_n$ the $n$-by-$n$ identity matrix.
We put
\[ \Gamma=\{ P \in SL(2,\Z) \ | P \equiv I_2 \mod 12 \}. \]

\begin{lem} \label{trivact}
	$\Gamma \subset \ker\rho$, that is, $\rho (P)=I_{10}$ for any $P\in\Gamma$.
\end{lem}
\begin{proof}
The GAP package Congruence  \cite{DJKV} shows that
$\Gamma$ is a normal closure of the set of the following elements:
\begin{xalignat*}{4}
 	P_{1,\pm} &=\mat1{\pm 12}01, & 
 	P_2 &=\mat{-143}{12}{-12}1, &
	P_3 &=\mat{-155}{84}{-24}{13}, 
	\\
	P_4 &=\mat{-191}{156}{-60}{49}, &
	P_5 &=\mat{-443}{120}{-48}{13},& 
	P_6 &=\mat{-467}{360}{-48}{37}, 
	\\
	P_7 &=\mat{-299}{108}{-36}{13},& 
	P_8 &=\mat{-311}{216}{-36}{25}, &
	P_9 &=\mat{937}{-396}{168}{-71},
	\\
	P_{10} &=\mat{157}{-36}{48}{-11}, &
	P_{11} &=\mat{157}{-48}{36}{-11}, & 
	P_{12} &=\mat{205}{-84}{144}{-59}, 
	\\
	P_{13} &=\mat{157}{-72}{24}{-11}, &
	P_{14} &=\mat{229}{-132}{144}{-83}, &
	P_{15} &=\mat{169}{-108}{36}{-23}, 
	\\ 
	P_{16} &=\mat{181}{-132}{48}{-35}, &
	P_{17} &=\mat{589}{-108}{60}{-11}, & 
	P_{18} &=\mat{649}{-384}{120}{-71}.
\end{xalignat*}
Further, we can verify that these matrices are  presented as 
the products of $S$ and $T$, as follows.
\begin{alignat*}{2}
&P_{1,\pm} =
	T^{\pm12}
	,&
&P_2 =
	S^2T^{12}ST^{12}S
 	,\\
&P_3 =
	S^2T^{7}ST^{2}ST^{7}ST^{2}S
	,&
&P_4 =
	T^{3}ST^{-5}ST^{2}ST^{-4}STS
	,\\
&P_5 =
	T^{9}ST^{-4}ST^{3}ST^{4}S
	,&
&P_6 =
	T^{10}ST^{4}ST^{3}ST^{-3}STS
	,\\
&P_7 =
	T^{8}ST^{-3}ST^{4}ST^{3}S
	,&
&P_8 =
	T^{9}ST^{3}ST^{4}ST^{-2}STS
	,\\
&P_{9} =
	T^{5}ST^{-2}ST^{-4}ST^{-4}ST^{-3}ST^{2}S,
	\\
&P_{10} =
	T^{3}ST^{-4}ST^{-3}ST^{4}S
	,&
&P_{11} =
	T^{4}ST^{-3}ST^{-4}ST^{3}S
	,\\
&P_{12} =
	TST^{-2}ST^{3}ST^{4}ST^{-2}ST^{2}S
	,&
&P_{13} =
	T^{6}ST^{-2}ST^{-6}ST^{2}S
	,\\
&P_{14} =
	TST^{-2}ST^{-3}ST^{4}ST^{4}ST^{2}S
 	,&
&P_{15} =
	S^2T^{4}ST^{3}ST^{-3}ST^{2}ST^{2}S
	,\\
&P_{16} =
	T^{4}ST^{4}ST^{-3}ST^{-3}STS
	,&
&P_{17} =
	S^2T^{10}ST^{5}ST^{-2}ST^{5}S
	,\\
&P_{18} =
	T^{5}ST^{-2}ST^{2}ST^{-4}ST^{3}ST^{2}S.
\end{alignat*} 
By using these formulae, we can verify that $\rho$ takes to each of the matrices to $I_{10}$,
completing the proof of the lemma.
\end{proof}

\begin{prop} \label{12periodic}
Let $p,q,p',q'$ be integers satisfying
$(p,q)=1$, $(p',q')=1$ and $p \equiv p' ,q \equiv q' \mod12$. 
Then,
$ Z(L(p,q))=Z(L(p',q')).$
\end{prop}

\begin{proof}
From Lemma \ref{lem.abab}, there exist matrices
\[
	A=\mat{-q}{b}{p}{-a},\ A'=\mat{-q'}{b'}{p'}{-a'} \in SL(2,\Z)
\]
such that $A \equiv A' \mod 12$.
We put
$P=I_2+A^{-1}(A'-A)$.
By definition, $P \in \Gamma$ and $A'=AP$.
Thus, by (\ref{eq.SW}) and Lemma \ref{trivact}, we have
\begin{align*}
	Z(L(p',q')) 
	& =w {}^t \e \rho(A')\e
	=w {}^t \e \rho(A)\rho(P) \e
	=w {}^t \e \rho(A)\e 
	=Z(L(p,q)),
\end{align*}
completing the proof.
\end{proof}

\begin{proof}[Proof of Theorem \ref{thm}]
	By Proposition \ref{12periodic}, 
	the left-hand side of (\ref{lensvalue}) has period $12$ for $p$ and $q$.
	On the other hand,
	the right-hand side of (\ref{lensvalue}) also has period 12 for $p$ and $q$.
	By \cite[Appendix D]{SW}, we can verify that (\ref{lensvalue}) holds 
	for coprime integers $p$ and $q$ with $1\le p \le 12$, $q<p$.
	Therefore,  (\ref{lensvalue}) holds for any coprime integers $p$ and $q$.
\end{proof}

\begin{proof}[Proof of Corollary \ref{cor}]
It is known, see \cite[12.1]{M}, that 
there exists an orientation-preserving homotopy equivalence
between the two lens spaces $L(p,q)$ and $L(p',q')$
if and only if $p=p'$ and $q  \equiv n^2 q' \mod p$ for some integer $n$.

When $(p,12) \neq 3,4,12$, by Theorem \ref{thm} the value $Z(p,q)$ does not depend on $q$. Thus, $Z(L(p,q))=Z(L(p',q'))$.

When $(p,12) = k$ with $k=3$ or $4$,
we have $q \equiv n^2 q' \mod k$ because $k|p$.
Thus, $q \not \equiv - q' \mod k$.
Thus, by Theorem \ref{thm},  $Z(L(p,q))=Z(L(p',q'))$.

When $12|p$, we have $q \equiv n^2 q' \mod 12$.
Thus, $q \not \equiv \pm5 q' \mod 12$.
Thus, by Theorem \ref{thm},  $Z(L(p,q))=Z(L(p',q'))$.
\end{proof}


\end{document}